\documentclass[final]{dmtcs-episciences}

\usepackage[table]{xcolor}
\usepackage[utf8]{inputenc}
\usepackage{subfigure}

\usepackage[T1]{fontenc}
\usepackage{amsmath,amssymb,amsthm}
\usepackage{thmtools}
\usepackage{cleveref}
\usepackage{hyperref}
\usepackage{url}
\usepackage{subcaption}
\usepackage{diagbox}

\def\paren#1{\left( #1 \right)}
\def\acc#1{\left\{ #1 \right\}}

\def\ceil#1{\left\lceil #1 \right\rceil}

\renewcommand{\le}{\leqslant}
\renewcommand{\ge}{\geqslant}

\theoremstyle{plain}
\newtheorem{theorem}{Theorem}
\newtheorem{lemma}[theorem]{Lemma}

\theoremstyle{definition}

\theoremstyle{remark}

\title{4-tangrams are 4-avoidable}
\author{Pascal Ochem\affiliationmark{1} \and Th\'eo Pierron\affiliationmark{2}}
\affiliation{
  LIRMM, CNRS, Universit\'e de Montpellier, Montpellier, France.\\
  Univ Lyon, UCBL, CNRS, INSA Lyon, LIRIS, Lyon, France.}

\begin{document}
\publicationdata{vol. 27:3}{2025}{1}{10.46298/dmtcs.15310}{2025-03-03; 2025-03-03; 2025-06-16}{2025-06-26}

\maketitle

\vspace{0mm}

\begin{abstract}
A tangram is a word in which every letter occurs an even number of times.
Thus it can be cut into parts that can be arranged into two identical words.
The \emph{cut number} of a tangram is the minimum number of required cuts in this process.
Tangrams with cut number one correspond to squares.
For $k\ge1$, let $t(k)$ denote the minimum size of an alphabet over which an infinite word 
avoids tangrams with cut number at most~$k$.
The existence of infinite ternary square-free words shows that $t(1)=t(2)=3$.
We show that $t(3)=t(4)=4$, answering a question from D\k{e}bski, Grytczuk, Pawlik, Przyby\l{}o, and \'Sleszy\'nska-Nowak.
\end{abstract}

\section{Introduction}\label{sec:intro}
A \emph{tangram} is a word in which every letter occurs an even number of times, possibly zero.
In particular, tangrams are a generalization of squares.
In this article, we consider a classification of tangrams depending on how close they are from being a square.
This relies on the so-called \emph{cut number} of a tangram, recently introduced by D\k{e}bski, Grytczuk, Pawlik, Przyby\l{}o, and \'Sleszy\'nska-Nowak~\cite{Debski(2024)}.
The \emph{cut number} of a tangram is defined as the minimum number of cuts needed so that the parts can be rearranged into two identical words.
Tangrams with cut number at most $k$ are called $k$-tangrams. Note that $1$-tangrams are exactly squares, and the larger the cut number, the farther the tangram is from a square.

Let $\Sigma_q=\acc{\texttt{0},\texttt{1},\ldots, q-1}$ denote the $q$-letter alphabet.
It is straightforward to check that every binary word of length 4 contains a square, while a famous theorem from Thue in 1906 asserts that there exist infinite words avoiding squares over $\Sigma_3$.
In~\cite{Debski(2024)}, the authors consider a similar question by investigating (infinite) words without tangrams.
By Lemma~3 in~\cite{Rampersad(2011)}, every infinite word must contain some tangram.
Thus, they consider the relation between the size of the alphabet and the cut number of the excluded tangrams.
For $k\ge1$, the authors thus define $t(k)$ as the minimum alphabet size such that there exists an infinite word 
avoiding $k$-tangrams. By definition, $t(k)$ is non-decreasing, and since every $2$-tangram contains a square, the result of Thue shows that $t(1)=t(2)=3$.

The other results in~\cite{Debski(2024)} are summarized in the following. 

\begin{theorem}[\cite{Debski(2024)}]\ 
\begin{itemize}
\item $t(k)\le1024\ceil{\log_2 k+\log_2 \log_2 k+1}$ for every $k\ge3$.
\item $t(k)\le k+1$ for every $k\ge4$.
\item $4\le t(3)\le t(4)\le5$
\end{itemize}
\end{theorem}
Moreover, the authors leave as an open problem the exact value of $t(3)$. In this article, we prove the following. 
\begin{theorem}
\label{thm:main}
$t(3)=t(4)=4$.
\end{theorem}

\section{Preliminaries}
To obtain Theorem~\ref{thm:main}, we heavily use the following relation observed in~\cite{Debski(2024)} between $k$-tangrams and patterns.
A pattern $P$ is a finite word over the alphabet $\Delta=\acc{A,B,\dots}$, whose letters are called \emph{variables}.
An \emph{occurrence} of a pattern $P$ in a word $w\in\Sigma^*$ is a non-erasing morphism $h:\Delta^*\to\Sigma^*$ such that $h(P)$ is a factor of $w$,
and a word $w$ avoids a pattern $P$ if it contains no occurrence of~$P$.
Given a square-free word $w$, a \emph{repetition} in $w$ is a factor of $w$ of the form $uvu$.
Its \emph{period} is $|uv|$ and its \emph{exponent} is $\tfrac{|uvu|}{|uv|}$.
Given $\alpha\in\mathbb{Q}$ and $n\in\mathbb{N}$, a word $w$ is \emph{$(\alpha^+,n)$-free} if it does not contain any repetition with period at least $n$ and exponent strictly greater than $\alpha$.
We say that $w$ is \emph{$\alpha^+$-free} if it is $(\alpha^+,1)$-free. 
A morphism $f:\Sigma_s^*\rightarrow\Sigma_e^*$ is \emph{$q$-uniform} if $|f(a)|=q$ for every $a\in\Sigma_s$.

As noticed in~\cite{Debski(2024)}, a $k$-tangram is an occurrence of some pattern with at most $k$ variables such that every variable occurs exactly twice.
So for every $k\ge1$, there exists a minimum set $S_k$ of such patterns such that avoiding
$S_k$ is equivalent to avoiding $k$-tangrams. Obviously, $S_k\subset S_{k+1}$ for every $k\ge1$.
A~small case analysis gives the first four sets $S_k$:
%
\begin{itemize}
\item $S_1=S_2=\acc{AA}$
\item $S_3=\acc{AA, ABACBC, ABCACB, ABCBAC}$
\item $S_4=\{AA, ABACBC, ABCACB, ABCBAC, ABACBDCD, ABACDBDC,\\
ABACDCBD, ABCACDBD, ABCADBDC, ABCADCBD, ABCADCDB,\\
ABCBADCD, ABCBDACD, ABCBDADC, ABCBDCAD, ABCDACBD,\\
ABCDADCB, ABCDBADC, ABCDBDAC, ABCDCADB, ABCDCBAD\}$,
\end{itemize}
In the next section, we prove Theorem~\ref{thm:main} by constructing infinite words over $\Sigma_4$ avoiding all patterns in $S_4$.
But first, let us show the weaker result $t(3)\le4$ as a straightforward (and computer-free) consequence of well-know results in pattern avoidance.
Following Cassaigne~\cite{Cassaigne(1994)}, we associate to each pattern a \emph{formula}, by replacing each variable appearing only once by a dot (such variables are called \emph{isolated}).
For example, the formula associated to the pattern $ABBACABADAA$ is $ABBA.ABA.AA$. The factors between the dots are called \emph{fragments}. Similarly to patterns, an \emph{occurrence} of a formula $f$ in a word $w\in\Sigma^*$ is a non-erasing morphism $h:\Delta^*\to\Sigma^*$
such that every fragment of $f$ is mapped under $h$ to a factor of $w$ (note that the order of the fragments does not matter). A word $w$ \emph{avoids} a formula $f$ if it contains no occurrence of $f$.

Consider the formula $F_3=AB.BA.AC.CA.BC$. Notice that the pattern $AA$, considered as a word, contains an occurrence of $F_3$.
Moreover, $ABACBC$, $ABCACB$, and $ABCBAC$ also contain an occurrence of $F_3$
since they have 5 distinct factors of length 2.
So every pattern in $S_3$ contains an occurrence of $F_3$.
Baker, McNulty, and Taylor~\cite{Baker(1989)} have considered the fixed point $b_4\in\Sigma_4^\omega$ of the morphism defined by
$\texttt{0}\mapsto\texttt{01}$, $\texttt{1}\mapsto\texttt{21}$, $\texttt{2}\mapsto\texttt{03}$, $\texttt{3}\mapsto\texttt{23}$
(that is, $b_4=\texttt{01210321012303210121}\ldots$) and shown that $b_4$ avoids $F_3$.
Then $b_4$ avoids every pattern in $S_3$. So $b_4$ avoids 3-tangrams, which implies that $t(3)\le4$.

\section{Proof of $t(4)\le 4$}\label{t_4}
Unfortunately, the word $b_4$ contains the factor \texttt{03210123} which is a 4-tangram.
Moreover, backtracking shows that every infinite word over $\Sigma_4$ avoiding 4-tangrams
must contain a factor $aba$ for some letters $a$ and $b$.
In particular, every $\tfrac75^+$-free word over $\Sigma_4$ contains a 4-tangram.
More generally, we have not been able to find a word that might witness $t(4)\le 4$ in the literature.
Thus we use an ad-hoc construction.
Consider the $312$-uniform morphism $h:\Sigma_6^*\to\Sigma_4^*$ below.

{\footnotesize
\begin{align*}
 &\texttt{0}\rightarrow\texttt{130321031012301020132102123120121302010310120301303213103132301030213032031}\\ &\texttt{2023020103212012312102132010203120121302123210131203013023031032013101230102103}\\ &\texttt{2120213012010310210132021230323010203121021321201231020103210212301213120321323}\\ &\texttt{0231232013121030103231302123120121321023132301232032103023031320103102101301203}\\ 
 &\texttt{1}\rightarrow\texttt{130321031012301020131210323130212320102031202302130320301023130132131031230}\\ &\texttt{1030213013120313231012130201203202102301210132010310231303120130102103101230313}\\ &\texttt{2131032312320121321202313230123203210302031230232132032310203012101232021303203}\\ &\texttt{1030230132023210320301230313020103212023132301232032103023031320103102101301203}\\ 
 &\texttt{2}\rightarrow\texttt{130321031012301020131210323130203012303103210131023013032031030213101201031}\\ &\texttt{2102123020120232101213021231203132131012303103203013023101312031030213032301020}\\ &\texttt{3121013202123032310232132013123130321013012010310213031320130102310121030201323}\\ &\texttt{0231232032130203012303103210131023013032031030213101232123031320103102101301203}\\ 
 &\texttt{3}\rightarrow\texttt{130321031012301020131210323130203012303103210130102031210213212012310201032}\\ &\texttt{1021230121312023210302301303203102320213023031203010323130212320131210301023101}\\ &\texttt{2013202102030123203213230231203020132032310232123013213103123130212320102013121}\\ &\texttt{0323123023213203121310231232013230321202302012320323103023031320103102101301203}\\ 
 &\texttt{4}\rightarrow\texttt{130321031012301020131210212310120103202102030123203213230231203020132032310}\\ &\texttt{2321230313201030212023102032013230232120313213013123103212320132131023130312101}\\ &\texttt{3202102302012032101213021020312023201030213101230313202321032312301213120321323}\\ &\texttt{0231232013121030103231302123120121321023132301232032103020312302321320323102030}\\ 
 &\texttt{5}\rightarrow\texttt{130321031012301020131210212310120103202102030123203213230231203020132021023}\\ &\texttt{1012010302131031231301321030102310313201312130231232032132301213102010232120313}\\ &\texttt{2130131231032123201321310231303120130102103101230313202320103021310312313013210}\\ &\texttt{3010231012013202120310210130120103212023132301232032103020312302321320323102030}
\end{align*}
}

We will show that for every $\tfrac65^+$-free word $w$ over $\Sigma_6$, 
$h(w)$ avoids every pattern in $S_4$.
Together with the result of~\cite{Kolpakov(2011)} that there exist
exponentially many $\tfrac65^+$-free infinite words over $\Sigma_6$, and the fact that $h$ is injective,
this implies that there exist exponentially many words over $\Sigma_4$ avoiding 4-tangrams.

First, we show that $h(w)$ is $\paren{\tfrac54^+,9}$-free by using the following lemma.
A uniform morphism $f:\Sigma^*\rightarrow\Delta^*$ is \emph{synchronizing}
if for all $a,b,c\in\Sigma$ and $u,v\in \Delta^*$, if $f(ab)=uf(c)v$, then either $u=\varepsilon$ and $a=c$, or $v=\varepsilon$ and $b=c$.
\begin{lemma}[\cite{Ochem(2006)}]\label{sync}
Let $\alpha,\beta\in\mathbb{Q},\ 1<\alpha<\beta<2$ and $n\in\mathbb{N}^*$.
Let $h:\Sigma^*_s\rightarrow\Sigma^*_e$ be a synchronizing $q$-uniform morphism (with $q\ge 1$).
If $h(w)$ is $\paren{\beta^+,n}$-free for every $\alpha^+$-free word $w$ such that
$|w|<\max\paren{\frac{2\beta}{\beta-\alpha},\frac{2(q-1)(2\beta-1)}{q(\beta-1)}}$,
then $h(t)$ is $\paren{\beta^+,n}$-free for every (finite or infinite) $\alpha^+$-free word $t$.
\end{lemma}
To see that $h$ is synchronizing, notice that $p=\texttt{13032103101230102013}$ is a prefix of the $h$-image of every letter and that $p$ does not appear elsewhere.
Moreover, we have checked that the $h$-image of every $\tfrac65^+$-free word of length smaller than $\frac{2\times\tfrac54}{\tfrac54-\tfrac65}=50$ is $\paren{\tfrac54^+,9}$-free.
Therefore $h(w)$ is $\paren{\frac54^+,9}$-free by Lemma~\ref{sync}. 


Now we show that every occurrence $m$ of a pattern $P\in S_4$ in a $\paren{\tfrac54^+,9}$-free word
is such that $|m(P)|$ is bounded (see Table 1).
%
As an example, let us detail the case of $ABCDACBD$. 
To lighten notations, we write $y=|m(Y)|$ for every variable~$Y$.
\begin{lemma}\label{lem:bounded}
Let $z$ be a $\paren{\tfrac54^+, 9}$-free word. Then if $z$ contains an occurrence $m$ of $ABCDACBD$, then $|m(ABCDACBD)|\le24$.
\end{lemma}

\begin{proof}
Consider an occurrence $m$ of $ABCDACBD$ in $z$. The factor $m(ABCDA)$ of $z$ is a repetition with period $|m(ABCD)|$ and exponent $\tfrac{|m(ABCDA)|}{|m(ABCD)|}$.

Since $z$ is $\paren{\tfrac54^+,9}$-free, then $a+b+c+d\le8$ or $\tfrac{2a+b+c+d}{a+b+c+d}\le\tfrac54$.
The latter inequality gives $\tfrac{a}{a+b+c+d}\le\tfrac14$ and then
\begin{equation}\label{i}
3a\le b+c+d.
\end{equation}
Similarly, the repetition $m(BCDACB)$ implies that $a+b+2c+d\le8$ or
\begin{equation}\label{ii}
3b\le a+2c+d.
\end{equation}
$m(CDAC)$ implies that $a+c+d\le8$ or
\begin{equation}\label{iii}
3c\le a+d.
\end{equation}
$m(DACBD)$ implies that $a+b+c+d\le8$ or
\begin{equation}\label{iv}
3d\le a+b+c.
\end{equation}
Suppose that $a+c+d\ge9$.
Then the combination $6\times\mbox{{\bf(\ref{i})}}+
4\times\mbox{{\bf(\ref{ii})}}+
7\times\mbox{{\bf(\ref{iii})}}+
6\times\mbox{{\bf(\ref{iv})}}$
gives $a+c+d\le0$, a contradiction. Therefore
\begin{equation}\label{v}
a+c+d\le8.
\end{equation}
This implies
\begin{equation}\label{vi}
c\le6.
\end{equation}
Now suppose that $b\ge5$, so that $a+b+2c+d\ge9$.
Then the combination $\mbox{{\bf(\ref{ii})}}+
\mbox{{\bf(\ref{v})}}+
\mbox{{\bf(\ref{vi})}}$
gives $3b\le14$, which contradicts $b\ge5$. Therefore
\begin{equation}\label{vii}
b\le4.
\end{equation}
By $\mbox{{\bf(\ref{v})}}$ and $\mbox{{\bf(\ref{vii})}}$, we get that $a+b+c+d\le 12$, hence $|m(ABCDACBD)|\le 24$.    
\end{proof}

Now, notice that every pattern in $S_4$ is \emph{doubled}, that is, every variable appears at least twice~\cite{Domenech(2023),Ochem(2016)}.
Alternatively, a doubled pattern is a formula with exactly one fragment.
The \emph{avoidability exponent} $AE(f)$ of a pattern or a formula $f$ is the largest real $x$ such that every $x$-free word avoids $f$.
By Lemma~10 in~\cite{Ochem(2021)}, the avoidability exponent of a doubled pattern with 4 variables
is at least $\tfrac65$. This bound is not good enough, so we have computed the avoidability exponent of every pattern in $S_4$, see Table 1. Notice that these avoidability exponents are greater than $\tfrac54$.
Then the $\paren{\tfrac54^+,9}$-freeness of $h(w)$ ensures that there is no "large" occurrence of a pattern
in $S_4$, that is, such that the period of every repetition is at least 9. This is witnessed by the combination
$6\times\mbox{{\bf(\ref{i})}}+
4\times\mbox{{\bf(\ref{ii})}}+
7\times\mbox{{\bf(\ref{iii})}}+
6\times\mbox{{\bf(\ref{iv})}}$ in the proof of Lemma~\ref{lem:bounded}.
Now, to bound the length of the other occurrences, we do not rely on a tedious analysis by hand as in
Lemma~\ref{lem:bounded}. Instead, the bound in the last column of Table 1 is computed as 
the maximum of $2(a+b+c+d)$ such that $1\le a, b, c, d<100$\footnote{We chose the bound 100 as it is sufficient and keeps the computer check short.} and
$(a+b+c+d\le8\vee3a\le b+c+d)\wedge(a+b+2c+d\le8\vee3b\le a+2c+d)\wedge(a+c+d\le8\vee3c\le a+d)\wedge(a+b+c+d\le8\vee3d\le a+b+c)$,
again with the example of $P=ABCDACBD$ of Lemma~\ref{lem:bounded}.

Finally, for every pattern $P\in S_4$, we check exhaustively by computer 
that $h(w)$ contains no occurrence of $P$ of length at most the corresponding bound\footnote{The C code to check the properties of the morphism $h$ and the bounds on $|m(P)|$ is available at\\ \url{http://www.lirmm.fr/~ochem/morphisms/tangram4.htm}.}.
So $h(w)$ avoids every $P\in S_4$. So $h(w)$ avoids 4-tangrams. So $t(4)\le4$.

\begin{table}[h]
\begin{center}
{\begin{tabular}{|l|l|l|l|}
\hline
Pattern $P$ & $P^R$ & $AE(P)$ & Bound on $|m(P)|$\\
\hline
$AA$ & self-reverse & 2 & 16\\
\hline
$ABACBC$ & self-reverse & $1.414213562=\sqrt{2}$ & 30\\
\hline
$ABCACB$ & $ABCBAC$ & 1.361103081 & 26\\
\hline
$ABACBDCD$ & self-reverse & 1.381966011 & 32\\
\hline
$ABACDBDC$ & $ABCBADCD$ & $1.333333333=\tfrac43$ & 40\\
\hline
$ABACDCBD$ & $ABCACDBD$ & 1.340090632 & 32\\
\hline
$ABCADBDC$ & $ABCBDACD$ & 1.292893219 & 32\\
\hline
$ABCADCBD$ & self-reverse & 1.295597743 & 28\\
\hline
$ABCADCDB$ & $ABCBDCAD$ & 1.327621756 & 32\\
\hline
$ABCBDADC$ & self-reverse & 1.302775638 & 32\\
\hline
$ABCDACBD$ & self-reverse & 1.258055872 & 24\\
\hline
$ABCDADCB$ & $ABCDCBAD$ & 1.288391893 & 42\\
\hline
$ABCDBADC$ & self-reverse & 1.267949192 & 24\\
\hline
$ABCDBDAC$ & $ABCDCADB$ & 1.309212406 & 44\\
\hline
\end{tabular}}
\end{center}
\caption{The patterns in $S_4$, their avoidability exponent, and the upper bound for the length of their occurrences in a $\paren{\tfrac54^+,9}$-free word.}
\end{table}

\section{Concluding remarks}\label{conclusion}
Our words $h(w)$ contain the factor $\texttt{012130212321}$
which is a 5-tangram since $\texttt{0|1|213021|2|3|21}$ can be rearranged as $\texttt{213021|2|1|3|0|21}$.
The exact value of $t(k)$ remains unknown for every $k\ge5$. In particular, we only known that $4\le t(5)\le6$.
Improving the upper bound on $t(5)$ using the approach in this paper might be tedious, as we expect
the set $S_5$ to be quite large.

\nocite{*}
\bibliographystyle{abbrv}
\bibliography{sample-dmtcs}
\label{sec:biblio}

\end{document}